\documentclass[12pt]{article}
\usepackage[utf8]{inputenc}
\usepackage{amsmath,amssymb}
\usepackage{amsthm} 
\usepackage{subfigure}
\usepackage{graphicx}
\usepackage{mathtools}
\usepackage{tikz}
\usetikzlibrary{arrows}
\usetikzlibrary{intersections}
\usepackage{tkz-euclide}
\newtheorem{theorem}{Theorem}[section]

\newtheorem{lemma}[theorem]{Lemma}
\newtheorem{corollary}[theorem]{Corollary}
\usepackage{hyperref}
\begin{document}
\title{\bf The facets of the spanning trees polytope}
\date{}
\maketitle
\begin{center}
\author{
{\bf Brahim Chaourar} \\
{ Department of Mathematics and Statistics,\\Imam Mohammad Ibn Saud Islamic University (IMSIU) \\P.O. Box
90950, Riyadh 11623,  Saudi Arabia }}
\end{center}


\begin{abstract}
\noindent Let $G=(V, E)$ be an undirected graph. The spanning trees polytope $P(G)$ is the convex hull of the characteristic vectors of all spanning trees of $G$. In this paper, we describe all facets of $P(G)$ as a consequence of the facets of the bases polytope $P(M)$ of a matroid $M$, i.e., the convex hull of the characteristic vectors of all bases of $M$.
\end{abstract}



\section{Introduction}

Sets and their characteristic vectors will not be distinguished. We refer to \cite{BondyMurty:2008}, \cite{Oxley:1992} and \cite{Schrijver:2004}, respectively, about graphs, matroids, and polyhedra terminology and facts. Readers who are familiar with matroid theory can skip the next two paragraphs.
\newline Given a finite set $E$, a matroid $M$ defined on $E$, is the pair $(E, \mathcal {B}(M))$ where $\mathcal {B}(M)$ is a nonempty class of subsets of $E$ satisfying the following basis exchange axiom: for any pair $(B_1, B_2)\in (\mathcal {B}(M))^2$ and for any $e\in B_1\backslash B_2$, there exists $f\in B_2\backslash B_1$, such that $B_1\cup \{ f\}\backslash \{ e\}\in \mathcal {B}(M)$. In this case, $E$ is the ground set of $M$ and $\mathcal {B}(M)$ is the class of bases of $M$. It follows that all bases of $M$ have the same cardinality. We can define the dual $M^*$ of $M$ as the matroid $(E, \mathcal {B}(M^*))$ where $\mathcal {B}(M^*)=\{ E\backslash B$ such that $B\in \mathcal {B}(M)\}$. The rank function of $M$, denoted by $r$, is a nonnegative integer function defined on $2^E$, the class of subsets of $E$, such that $r(X)=Max\{ |X\cap B|$ for all $B\in \mathcal {B}(M)\}$ for any $X\subseteq E$. The rank of $M$, denoted by $r(M)$, is $r(E)$, and it is equal to the cardinality of any basis. The rank function of $M^*$, denoted by $r^*$, is called the dual rank function of $M$. It is not difficult to see that $r^*(X)=|X|-r(E)+r(E\backslash X)$ for any $X\subseteq E$. Finally, $P(M)$ is the convex hull of all bases of $M$, and it is called the bases polytope of $M$.
\\ Given $e\in E$, we can define two operations. The first one is called the deletion of $e$, denoted by $M\backslash e$, and defined by $M\backslash e=(E\backslash \{ e\}, \mathcal {B}(M\backslash e))$, where $\mathcal {B}(M\backslash e)=\{ B\in \mathcal {B}(M)$ such that $e\notin B\}$. The second one is called the contraction of $e$, denoted by $M/e$, and defined by $M/e=(M^*\backslash e)^*$, i.e., the matroid dual of the deletion of $e$ in the dual, which means that $\mathcal {B}(M/e)=\{ B\backslash \{ e\}$ such that $B\in \mathcal {B}(M)$ and $e\in B\}$. For any subset $X\subseteq E$, we denote by $M\backslash X$ (respectively, $M/X$), the matroid obtained by doing successive deletions (respectively, contractions) of all elements of $X$. A matroid $N$ is called a minor of $M$ if it is obtained by successive operations of deletion and/or contraction. In this case, we denote by $E(N)$, the ground set of $N$. A subset $X\subseteq E$ is called closed if $r(X\cup \{ e\})=r(X)+1$ for any $e\in E\backslash X$. Moreover, $M|X$ is the matroid defined by $(X, \mathcal {B}(X))$, where $\mathcal {B}(X)=\{ B\cap X$ such that $B\in \mathcal {B}(M)$ and $|B\cap X|=r(X)\}$. In other words, $M|X=M\backslash (E\backslash X)$, and $M^*|(E\backslash X)=(M/X)^*$. A matroid $M$ is disconnected if there exists a proper nonempty subset $X\subset E$ such that $r(E)=r(X)+r(E\backslash X)$. We say that $M$ is 2-connected if it is not disconnected. It is not difficult to see that $M$ is 2-connected if and only if $M^*$ is too. A minor $N$ of $M$ is a 2-connected component of $M$ if $N$ is 2-connected and $E(N)$ is maximal (by inclusion) for this property. The direct sum of two matroids $M_1$ and $M_2$, denoted by $M_1\oplus M_2$, is defined by the matroid $(E(M_1)\cup E(M_2), \mathcal{B}(M_1)\times \mathcal{B}(M_2))$. It is not difficult to see that any disconnected matroid is the direct sum of its 2-connected components.
\\Suppose that $M$ (and $M^*$) is 2-connected. A subset $L\subset E$ is called a locked subset of $M$ if $M|L$ and $M^*|(E\backslash L)$ are 2-connected, and their corresponding ranks are at least 2, i.e., $min\{r(L), r^*(E\backslash L)\} \geq 2$. In other words, $L$ is locked in $M$ if and only if $M|L$ and $M/L$ are 2-connected, and $r(L)\geq max\{2, 2+r(E)-|E\backslash L|\}$. It is not difficult to see that if $L$ is locked then both $L$ and $E\backslash L$ are closed, respectively, in $M$ and $M^*$ (That is why we call it locked). For a disconnected matroid $M$, locked subsets are unions of locked subsets in the 2-connected components of $M$. Locked subsets were introduced in \cite{Chaourar:2002, Chaourar:2008, Chaourar:2011, Chaourar:2018} to solve many combinatorial problems in matroids.
\\ A parallel closure $P\subseteq E$ of a matroid $M$ defined on $E$, is a maximal subset such that $r(P)=1$ (and any $e\in P$ satisfy $r(e)=1$). A coparallel closure $S\subseteq E$ of $M$ is a parallel closure of the dual $M^*$. In the case of a 2-connected matroid $M$, a parallel closure $P$ (respectively, coparallel closure $S$) is essential if $M/P$ (respectively, if $M\backslash S$) is 2-connected.
\\ For any undirected graph $G$, and any subgraph $H$ of $G$, we denote by $V(H)$ (respectively, $E(H)$), the set of vertices (respectively, edges) of $H$. For any subset of vertices $U\subseteq V(G)$, we denote by $G(U)$ (respectively, $E(U)$) the induced subgraph of $G$ based on the vertices of $U$ (respectively, $E(G(U))$). For any subset of edges $F\subseteq E(G)$, we denote by $V(F)$ the set of vertices incident to any edge of $F$. We also use the notations: $n=|V(G)|$, $m=|E(G)|$, and $n_H=|V(H)|$, $m_H=|E(H)|$ for any subgraph or any subset of edges $H$ of $G$. For any subset $F\subseteq E$, and any $x\in \mathbb{R}^E$, $x(F)=\sum\limits_{e\in F} x(e)$.
\\ For an induced subgraph $H$ of $G$, $\overline{H}=(V(E(G)\backslash E(H)), E(G)\backslash E(H))$ is called the complementary subgraph of $H$ in $G$, i.e., the subgraph obtained by removing the edges of $H$ and vertices which are not incident to any edge of $E(G)\backslash E(H)$. Moreover, for any $F\subseteq E(G)$, $H\cup F$ is the subgraph $(V(H)\cup V(F), E(H)\cup F)$.
\\ Matroids generalize graphs. Given an undirected connected graph $G$, we can define the corresponding (graphical) matroid $M(G)$ as the pair $(E(G), \mathcal T(G))$, where $\mathcal T(G)$ is the class of spanning trees of $G$. By analogy, $P(G)=P(M(G))$ is the spanning trees polytope, i.e., the convex hull of all spanning trees of $G$, where $M(G)$ is the corresponding (graphical) matroid of $G$. A locked subgraph $H$ of $G$ is a subgraph for which $E(H)$ is a locked subset of $M(G)$. Deletions and contractions in the corresponding (graphical) matroid $M(G)$ correspond to classical deletions and contractions in $G$. A minor $N$ of $M(G)$ correspond to the graphical matroid $M(H)$ of a minor $H$ of $G$. A graphical matroid is 2-connected if and only if its corresponding graph is 2-connected. For any connected induced subgraph $G(U)$ of $G$, $r(E(U))=|U|-1$.
\newline In this paper, we give a minimal description of $P(G)$ by means of graph theory. We suppose that considered matroids and graphs are 2-connected because $P(M)$ (respectively, $P(G)$) is the cartesian product of the corresponding polytopes in the 2-connected components.
\newline Schrijver claimed \cite{Schrijver:2004} (page 862, discussion after Corollary 50.7d), and referring to a result of Gr\"otschel \cite{Grotschel:1977}, that the nontrivial facets of $P(G)$ are described by induced and 2-connected subgraphs as for the forests polytope, i.e., the convex hull of all forests of $G$ (a minimal description of the forests polytope was done in \cite{Pulleyblank:1989} as the set of all $x\in \mathbb{R}
^E$ satisfying: $x(e)\geq 0$ for any edge $e$, and $x(E(U))\leq |U|-1$ for any $U\subseteq E$ inducing a 2-connected subgraph with $|U|\geq 2$). In this paper, we show that some further assumptions are needed (see Theorem 2.3). We present a counterexample to Schrijver's claim at the end of section 2.
\newline The remainder of the paper is organized as follows: in section 2, we give a minimal description of $P(G)$, then we give alternative minimal decsriptions of $P(M)$ and $P(G)$ in section 3. Finally, we conclude in section 4.


\section{Facets of the spanning trees polytope}

A minimal description of $P(M)$ has been given independently in \cite{Fujishige:1984}, \cite{FeichtnerSturmfels:2005}, \cite{Chaourar:2018}, and \cite{HibiEtAl:2019} (see \cite{Kolbl:2020}) as follows.

\begin{theorem} A minimal description of $P(M)$ is the set of all $x\in \mathbb{R}^E$ satisfying the following constraints:
$$    x(P) \leq 1      \> for\>any \> essential\> parallel\> closure\> P\subseteq E               \eqno       (1)$$
$$    x(S) \geq |S|-1  \> for\> any\> essential\>  coparallel\> closure\> S\subseteq E            \eqno       (2)$$
$$    x(L) \leq r(L)   \> for\> any\> locked\> subset\> L\subseteq E                              \eqno       (3)$$
$$    x(E)=r(E)                                                                                   \eqno       (4)$$
\end{theorem}

For the graphical case, a parallel closure is the edge set of an induced subgraph on two vertices, and a coparallel closure is a series closure, i.e., a maximal set of edges forming a simple path for which all involved vertices except its two terminals have degree 2. An essential parallel closure is the edge set of an induced subgraph on two vertices whose contraction keep the graph 2-connected. An essential coparallel closure is the edge set of a series closure whose deletion keep the graph 2-connected. It remains to translate lockdness in graphical terms.
\newline First, we prove the following lemma.

\begin{lemma}\label{ConnectedComplement} Let $H$ be a 2-connected subgraph of $G$, and $\{ L_1, L_2\}$ be a partition of $E(\overline{H})$ such that $(V(L_i), L_i)$ is connected, $i=1, 2$. Then $G(V(G)\backslash V(H))$ is connected if and only if $n_H+n<n_{H\cup L_1}+n_{H\cup L_2}$.
\end{lemma}
\begin{proof} $$ n_H+n<n_{H\cup L_1}+n_{H\cup L_2}$$
$ \Longleftrightarrow n_H+n_H+|V(\overline{H})|-|V(H)\cap V(\overline{H})|<n_H+n_{L_1}-|V(H)\cap V(L_1)|+n_H+n_{L_2}-|V(H)\cap V(L_2)|$\\
$ \Longleftrightarrow 2n_H+|V(\overline{H})|-|V(H)\cap V(\overline{H})|<2n_H+n_{L_1}+n_{L_2}-|V(H)\cap V(L_1)|-|V(H)\cap V(L_2)|$\\
$ \Longleftrightarrow 2n_H+|V(\overline{H})|-|V(H)\cap V(\overline{H})|<2n_H+n_{L_1}+n_{L_2}-|V(H)\cap V(\overline{H})|-|V(H)\cap V(L_1)\cap V(L_2)|$\\
$ \Longleftrightarrow |V(\overline{H})|<n_{L_1}+n_{L_2}-|V(H)\cap V(L_1)\cap V(L_2)|$\\
$ \Longleftrightarrow n_{L_1}+n_{L_2}-|V(L_1)\cap V(L_2)|<n_{L_1}+n_{L_2}-|V(H)\cap V(L_1)\cap V(L_2)|$\\
$ \Longleftrightarrow |V(L_1)\cap V(L_2)|>|V(H)\cap V(L_1)\cap V(L_2)|$\\
$ \Longleftrightarrow |V(L_1)\cap V(L_2)|\geq |V(H)\cap V(L_1)\cap V(L_2)|+1$\\
$ \Longleftrightarrow |(V(L_1)\cap V(L_2))\backslash V(H)|\geq 1$\\
which means that $G(V(G)\backslash V(H))$ is connected.
\end{proof}

\noindent Now we can characterize locked subgraphs by means of graphs terminology.

\begin{theorem}\label{LockedCharacterization} $H$ is a locked subgraph of $G$ if and only if $H$ is an induced and 2-connected subgraph such that $3\leq n_H\leq n-1$, $m_{\overline{H}}\geq n_{\overline{H}}$ or $|V(H)\cap V(\overline{H})|\geq 3$, and $G(V(G)\backslash V(H))$ is a connected subgraph.
\end{theorem}
\begin{proof}
Without loss of generality, we can suppose that $G$ is 2-connected.
\newline It is not difficult to see that $E(H)$ is closed and 2-connected in $M(G)$ if and only if $H$ is an induced and 2-connected subgraph of $G$.
\newline Now, suppose that $E(H)$ is closed and 2-connected in $M(G)$, and $E(G)\backslash E(H)$ is 2-connected in the dual matroid $M^*(G)$ (i.e., $E(H)$ is locked in $M(G)$). Let $\{ L_1, L_2\}$ be a partition of $E(G)\backslash E(H)$ such that the subgraph $(V(L_i), L_i)$ is connected, $i=1, 2$. It follows that $r^*(E(G)\backslash E(H))<r^*(L_1)+r^*(L_2)$, i.e., $|E(G)\backslash E(H)|-r(E(G))+r(E(H))<|L_1|+|L_2|-2r(E(G))+r(E(H)\cup L_1)|+r(E(H)\cup L_2)|$. In other words, $r(E(H))+r(E(G))<r(E(H)\cup L_1)+r(E(H)\cup L_2)$, which is equivalent to: $n_H-1+n-1<n_{H\cup L_1}-1+n_{H\cup L_2}-1$, i.e., $G(V(G)\backslash V(H))$ is connected according to the above lemma.
\newline Let check the condition: $min\{r(E(H)), r^*(E(\overline{H}))\} \geq 2$. Since $r(E(H))=n_H-1$, we have $r(E(H))\geq 2$ if and only if $n_H\geq 3$. Moreover, $r^*(E(G)\backslash E(H))=m_{\overline{H}}+r(E(H))-r(E(G))=m_{\overline{H}}+n_H-n$ then we have $r^*(E(G)\backslash E(H))\geq 2$ if and only if $n_H\geq 2+n-m_{\overline{H}}$, i.e., $|V(H)\cap V(\overline{H})|\geq 2+n_{\overline{H}}-m_{\overline{H}}$ (inequality (*)). But, if $G$ is 2-connected and $G(V(G)\backslash V(H))$ is connected, then $|V(H)\cap V(\overline{H})|\geq 2$, and $m_{\overline{H}}-|V(H)\cap V(\overline{H})|\geq |V(\overline{H})\backslash V(H)|-1\geq n_{\overline{H}}-|V(H)\cap V(\overline{H})|-1$, i.e., $m_{\overline{H}}\geq n_{\overline{H}}-1$. Thus, we have either $m_{\overline{H}}\geq n_{\overline{H}}$  or $m_{\overline{H}}=n_{\overline{H}}-1$.
\\ {\bf Case 1:} If $m_{\overline{H}}\geq n_{\overline{H}}$ then $2+n_{\overline{H}}-m_{\overline{H}}\leq 2$.
\\ {\bf Case 2:} If $m_{\overline{H}}=n_{\overline{H}}-1$ then $\overline{H}$ is a tree and $2+n_{\overline{H}}-m_{\overline{H}}=3$. If $|V(H)\cap V(\overline{H})|=2$ then $E\backslash E(H)=E(\overline{H})$ is a series closure. Hence $E\backslash E(H)$ is a parallel closure in the dual of $M(G)$ and $r^*(E\backslash E(H))=1$. It follows that $H$ is not locked.
\\Therefore, in both cases, the inequality (*) is equivalent to $m_{\overline{H}}\geq n_{\overline{H}}$ or $|V(H)\cap V(\overline{H})|\geq 3$.
\newline Furthermore, $E(H)$ is closed in $M(G)$ and distinct from $E$, i.e., $r(E(H))\leq r(E(G))-1$, which is equivalent to: $n_H\leq n-1$.
\end{proof}

So the consequence for the spanning tree polytope is:

\begin{corollary} A minimal description of $P(G)$ is the set of all $x\in \mathbb{R}^{E(G)}$ satisfying the following constraints:
$$    x(P) \leq 1          \> for\>any \> essential\> parallel\> closure\> P \>of \> G   \eqno        (5)$$
$$    x(S) \geq |S|-1      \> for\> any\> essential\> coparallel\> closure\> S \>of \> G \eqno        (6)$$
$$    x(E(H)) \leq n_H-1   \>for\> any\> locked\> subgraph\> H \>of \> G                 \eqno        (7)$$
$$    x(E(G))=n-1                                                                        \eqno        (8)$$
\end{corollary}

Note that many equivalent minimal descriptions can be given because $P(G)$ (as $P(M)$) is not full-dimensional ($dim(P(G))=|E(G)|-1$ if $G$ is 2-connected). We discuss this issue in section 3.
\\A direct consequence of the above corollary is the following corollary correcting the well-know idea about nontrivial facets of the spanning trees polytope.

\begin{corollary} The following constraint of $P(G)$ is redundant if $H$ is not locked:
$$    x(E(H)) \leq n_H-1    \> for \> any \> induced \> and \> 2-connected \> subgraph \> H \> of \> G           \eqno        (9)$$
\end{corollary}
\begin{proof}
Let $H$ be an induced and 2-connected subgraph of $G$ which is not locked. According to Theorem \ref{LockedCharacterization} and Lemma \ref{ConnectedComplement}, and given a partition $\{ L_1, L_2\}$ of $E(G)\backslash E(H)$ such that $(V(L_i), L_i)$ is connected, $i=1, 2$, we have $n_H+n\geq n_{H\cup L_1}+n_{H\cup L_2}$.
\newline In this case, for $x\in P(G)$, we have: $x(E(H))+x(E(G))=x(E(H)\cup L_1))+x(E(H)\cup L_2))\leq n_{H\cup L_1}+n_{H\cup L_2}-2\leq n_H+n-2$. But $x(E(G))=n-1$, hence $x(E(H))\leq n_H-1$.
\end{proof}

Next we give here an example of an induced and 2-connected subgraph which is not locked.
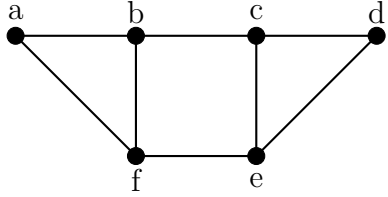
\begin{center}
\begin{tikzpicture}
  [scale=.8,auto=center]

\filldraw [black] (2,0) circle (4pt);

\filldraw [black] (4,0) circle (4pt);

\filldraw [black] (0,2) circle (4pt);

\filldraw [black] (2,2) circle (4pt);

\filldraw [black] (4,2) circle (4pt);

\filldraw [black] (6,2) circle (4pt);

\draw[black, thick] (2,0) -- (4,0);
\draw[black, thick] (2,0) -- (0,2);
\draw[black, thick] (2,0) -- (2,2);
\draw[black, thick] (4,0) -- (4,2);
\draw[black, thick] (4,0) -- (6,2);
\draw[black, thick] (0,2) -- (2,2);
\draw[black, thick] (2,2) -- (4,2);
\draw[black, thick] (4,2) -- (6,2);

\node at (0,2.4) {\text{a}};
\node at (2,2.4) {\text{b}};
\node at (4,2.4) {\text{c}};
\node at (6,2.4) {\text{d}};
\node at (4,-0.4) {\text{e}};
\node at (2,-0.4) {\text{f}};

\end{tikzpicture}
\end{center}

Let $H=G(\{ b, c, e, f\})$, i.e., the circuit $bcefb$, $L_1=\{ ab, af\}$, and $L_2=\{ dc, de\}$. $H$ induces a 2-connected subgraph which is not locked because $n_H+n=4+6=10\geq 10=5+5=n_{H\cup L_1}+n_{H\cup L_2}$.
\newline Note that this idea happened because it was thought that facets of the forests polytope are kept for one of its faces which is $P(G)$.

\section{Alternative equivalent minimal descriptions of $P(M)$ and $P(G)$}

In this section, we give alternative equivalent
minimal descriptions of $P(M)$ and $P(G)$.
\\ First we introduce some notations. We denote by $\mathcal P(M)$, $\mathcal S(M)$, and $\mathcal L(M)$, the class of, respectively, essential parallel closures, essential coparallel closures, and locked susbets of $M$.
\\ Now we consider the following constraints:
$$    x(E\backslash P) \geq r(E)-1            \> for\>any \> essential\> parallel\> closure\> P\subseteq E        \eqno       (9) $$
$$    x(E\backslash S) \leq r(E\backslash S)  \> for\> any\> essential\>  coparallel\> closure\> S\subseteq E     \eqno       (10)$$
$$    x(E\backslash L) \geq r(E)-r(L)         \> for\> any\> locked\> subset\> L\subseteq E                       \eqno       (11)$$

We need the following lemma.

\begin{lemma}
  If $S\in \mathcal S(M)$ then $r(E\backslash S)=r(E)-|S|+1$
\end{lemma}
\begin{proof}
  Since $r^*(S)=1$ and $r(E)=|E|-r^*(E)$, we have: $r(E\backslash S)=|E\backslash S|-r^*(E)+r^*(S)=|E|-|S|-r^*(E)+1=r(E)-|S|+1$, and we are done.
\end{proof}

\begin{theorem} Let $\mathcal P\subseteq \mathcal P(M)$, $\mathcal S\subseteq \mathcal S(M)$, and $\mathcal L\subseteq \mathcal L(M)$. A minimal description of $P(M)$ is the set of all $x\in \mathbb{R}^E$ satisfying the constraint (5), and the following constraints:
  \\ Constraint (1) for any $P\in \mathcal P$, and constraint (9) for any $P\in \mathcal P(M)\backslash \mathcal P$,
  \\ Constraint (2) for any $S\in \mathcal S$, and constraint (10) for any $S\in \mathcal S(M)\backslash \mathcal S$,
  \\ Constraint (3) for any $L\in \mathcal L$, and constraint (11) for any $L\in \mathcal L(M)\backslash \mathcal L$.
\end{theorem}
\begin{proof}
  It suffices to see that the corresponding pairs of constraints ((1) and (9), (2) and (10), (3) and (11)) are equivalent if constraint (5) is satisfied ($r(E)=x(E)=x(F)+x(E\backslash F)$ for any, respectively, essential parallel closure, essential coparallel closure, or locked subset $F$), and by using the above lemma for the pair (2) and (10).
\end{proof}

And similarly, we have alternative equivalent descriptions of $P(G)$ by translating matroid theory terms to graphs terminology as it was characterized in section 2.

\section{Conclusion}

We have described all facets of $P(G)$ correcting a well-known idea about nontrivial ones of them.

\vspace{11pt}
{\Large\bf Acknowledgements}

The author is grateful to Smail Djebali for his valuable remarks in a previous version of this paper.


\end{document}